\documentclass[12pt]{article}

\usepackage{amsmath,amssymb, comment}

\usepackage{graphicx}
\usepackage{epstopdf}
\usepackage{hhline}
\usepackage{color}
\usepackage{epsfig}

\newcommand{\bu}{{\mathbf u}}
\newcommand{\bv}{{\mathbf v}}
\newcommand{\be}{{\mathbf e}}
\newcommand{\bw}{{\mathbf w}}

\newcommand{\D}{{\cal D}}
\newcommand{\LL}{{\cal L}}

\newcommand{\R}{{\Bbb R}}
\newcommand{\sgn}{{\rm sign }}

\newcommand{\kron}{\otimes}

\newcommand{\tspan}{\mbox{span}}

\newcommand{\MATLAB}{{\small MATLAB }}

\newtheorem{definition}{Definition}
\newtheorem{algorithm}{Algorithm}
\newtheorem{theorem}{Theorem}

\newenvironment{proof}{\begin{list}{$\!\!${\bf Proof} \rule{1pt}{0pt}}
{\setlength{\leftmargin}{0pt}\setlength{\itemindent}{30pt}\setlength{\listparindent}
{15pt}}\item}{\rule{0.3em}{0mm}\hfill\framebox[1.2ex]{\rule{0.3em}{0mm}}
\end{list}}

\title{Accurate Inverses for Computing Eigenvalues of Extremely Ill-conditioned Matrices and Differential Operators
}

\author{Qiang Ye\thanks{Department of Mathematics, University of Kentucky,
Lexington, KY 40506. {\tt qye3@uky.edu}. Research supported in part by NSF Grants DMS-1317424, DMS-1318633 and DMS-1620082.
}
}
\date{}

\begin{document}

\maketitle

\begin{abstract}
This paper is concerned with  computations of  a few smallest eigenvalues  (in absolute value) of a large extremely ill-conditioned matrix. It is shown that a few smallest eigenvalues can be accurately computed for a diagonally dominant matrix or  a product of diagonally dominant matrices by combining a standard iterative method with the accurate inversion algorithms that have been developed for such matrices. Applications to the finite difference discretization of   differential operators are discussed. In particular, a new discretization is derived for the 1-dimensional biharmonic operator that can be written as a product of diagonally dominant matrices. Numerical examples are presented to demonstrate the accuracy achieved by the new algorithms.
\end{abstract}

%
%

\pagestyle{myheadings}
\thispagestyle{plain}

\section{Introduction}\label{sec:intro}
In this paper, we are concerned with accurate computations of  a few smallest eigenvalues  (in absolute value) of a large extremely ill-conditioned matrix. Here, we consider a matrix $A$ extremely ill-conditioned if $\bu \kappa_2 (A)$ is almost of order 1, where $\bu$ is the machine roundoff unit, $\kappa_2 (A):= \|A\| \|A^{-1}\|$ is the spectral condition number of $A$, and $\|\cdot\|$ is the 2-norm. We are mainly interested in large
sparse matrices arising in discretization of differential operators, which may lead  to extremely ill-conditioned matrices when a very fine discretization is used. In that case,  existing eigenvalue algorithms may compute a few smallest eigenvalues (the lower end of the spectrum) with little or no accuracy owing to roundoff errors in a floating point arithmetic.

Consider an $n\times n$  symmetric positive definite matrix $A$ and
let $0<\lambda_1 \le \lambda_2 \le \cdots \le \lambda_n$ be its eigenvalues.
Conventional dense matrix eigenvalue algorithms
(such as the QR algorithm)
are normwise backward stable, i.e., the
computed eigenvalues $\{\widehat \lambda_i \}$ in a floating point arithmetic
are the exact eigenvalues
of $A+E$ with $\|E\| = {\cal O} (\bu) \|A\|$;  see \cite[p.381]{gova}.
Here and throughout, 
${\cal O} (\bu)$ denotes a term bounded 
by $p(n)\bu$ for some polynomial $p(n)$ in $n$.
Eigenvalues of large (sparse) matrices are typically computed by
an iterative method (such as the Lanczos algorithm), which
produces
an approximate eigenvalue $\widehat \lambda_i $ and an approximate eigenvector
$\widehat x_i $ whose residual satisfies $\|A \widehat x_i - \widehat \lambda_i \widehat x_i\|/\|\widehat x_i\| < \eta$ for some threshold $\eta$. Since the roundoff error occurring in computing $A \widehat x_i$ is of order $\bu \|A\| \|\widehat x_i\| $, then this residual can converge at best to
${\cal O} (\bu) \|A\| $. 
Then, for both the dense and iterative eigenvalue algorithms, the error of the computed eigenvalue $\widehat\lambda_i$ is at best
$|\widehat\lambda_i - \lambda_i | \le {\cal O} (\bu) \|A\|= {\cal O} (\bu) \lambda_n$. Thus
\begin{equation}\label{eq:relerror}
\frac{|\widehat\lambda_i - \lambda_i |}{\lambda_i }
\le {\cal O} (\bu) \frac{\lambda_n }{ \lambda_i}
\approx
  \begin{cases}
{\cal O} (\bu)  & \text{if}\;\; \lambda_i\approx \lambda_n\\
{\cal O} (\bu) \kappa_2 (A) & \text{if}\;\; \lambda_i\approx \lambda_1.
  \end{cases}
\end{equation}
It follows that larger eigenvalues
(i.e. those $\lambda_i\approx \lambda_n$) can be computed with a relative error of  order $\bu$, but
for smaller eigenvalue (i.e. those $\lambda_i\approx \lambda_1$), we may expect a relative error of order ${\cal O} (\bu) \kappa_2 (A)$.
Hence, little   accuracy may be expected of these smaller eigenvalues  if the matrix  $A$ is extremely ill-conditioned. This is the case regardless of the magnitude of $\lambda_1$.

%

Large matrices arising from applications are typically inherently ill-conditioned.
Consider  the eigenvalue problems for a differential operator $\LL$.
Discretization of $\LL$
leads to a large and sparse matrix eigenvalue problem.
Here it is usually  a few smallest eigenvalues that
are of interest and are well approximated by
the discretization.
Then, as the discretization mesh size $h$ decreases,
the condition number increases
and then the relative accuracy of these smallest eigenvalues
as computed by existing algorithms deteriorates.
Specifically, the condition numbers  of finite difference discretization are typically of order
 ${\cal O} (h^{-2})$ for second order differential operators, but for fourth order  operators, it is
 of order ${\cal O} (h^{-4})$. This also holds true for other discretization methods;
see \cite{ben08,brms,fish,jiang14,zhxu14} and the references contained therein for some recent discussions on discretization of fourth order operators.
Thus, for a fourth order operator, little accuracy may be expected of the computed eigenvalues when $h$ is near $10^{-4}$ in the standard double precision (see numerical examples in \S 4).

Indeed, computing eigenvalues and eigenvectors of
a biharmonic operator $\Delta^2$ has been a subject of much discussion;
see \cite{bare,bjtj97,brms,brow,ciar,chen,chen91,hack,rann}.
It has been
noted by Bjorstad and Tjostheim \cite{bjtj99} that several earlier
numerical results obtained in \cite{bare,chen91,hack} based on coarse discretization schemes
are inaccurate in the sense that they either
miss or misplace some known nodal lines for the first eigenfunction.
Indeed, to obtain more accurate numerical results that
agree with various known theoretical properties of eigenvalues and eigenfunctions,
they   had to use the  quadruple precision to compute the eigenvalues of the matrix obtained from the spectral Legendre-Galerkin method with up to $5000$ basis functions in \cite{bjtj99}. Clearly,  the existing matrix eigenvalue algorithms implemented in
the standard double precision  could not provide satisfactory accuracy at this resolution.

The computed accuracy of  smaller eigenvalues of a matrix has   been discussed extensively in
the context of the
dense matrix eigenvalue problems in the last two decades.
Starting with a work by Demmel and Kahan \cite{demmelkahan}
on computing singular values of bidiagonal matrices, there is a large body of literature on when and how smaller eigenvalues or singular values can be computed accurately. Many special classes of
matrices have been identified
for which the singular values (or eigenvalues)
are determined and can be computed to high relative accuracy (i.e. removing
$\kappa_2 (A) $ in (\ref{eq:relerror})); see   \cite{barlowdemmel,ddy14a,ddy14b,demmel99a,demmel99b,demmelkoev,deve,doko,Dopico:09,fepa,li98} and the recent survey \cite{drma14} 
for more details.

In this paper, we study a new approach of using an iterative algorithm combined with the inverse (or the shift-and-invert transformation more generally)  to accurately compute a few smallest eigenvalues of a diagonally dominant matrix or one that admits factorization into a product of diagonally dominant matrices. A key observation is that those smallest eigenvalues can be accurately computed from the corresponding largest eigenvalues of the inverse matrix, provided the inverse operator can be computed {\em accurately}. In light of extreme ill-conditioning that we assume for the matrix, accurate inversion is generally not possible but, for diagonally dominant matrices, we can use the accurate $LDU$ factorization that we recently developed, with which the inverse (or linear systems) can be solved sufficiently accurately.
As applications, we will show that we can compute a few smallest eigenvalues accurately for most second order self-adjoint differential operators and some fourth order differential operators in spite of extreme ill-conditioning of the discretization matrix. Numerical examples will be presented to demonstrate the accuracy achieved.

We note that using the Lanczos algorithm with the shift-and-invert transformation is a standard  way for solving differential operator eigenvalue problems, where the lower end of the spectrum is sought but is clustered. The novelty  here is the use of the accurate $LDU$ factorization, that has about the same computational complexity as the Cholesky factorization but yields sufficiently accurate applications of the inverse operator.

The paper is organized as follows.  We discuss in \S 2 an accurate $LDU$ factorization algorithms for diagonally dominant matrices and then in \S 3, we show how they can be used to compute a few smallest eigenvalues accurately. We then apply this to discretization of differential operators together with various numerical examples in \S 4. We conclude with some remarks in \S 5.

Throughout, $\| \cdot \|$ denotes the 2-norm for vectors and matrices, unless  otherwise specified.
Inequalities and absolute value involving matrices and vectors are entrywise.  $\bu$ is the machine roundoff unit and ${\cal O} (\bu)$ denotes a term bounded by $p(n)\bu$ for some polynomial $p(n)$ in $n$. We use $fl(z)$ to denote the computed result of an algebraic expression $z$. ${\cal R} (M)$ denote the range space of a matrix $M$ and $\kron $ is the Kronecker product.

\section{Accurate $LDU$ Factorization of Diagonally Dominant Matrices}

Diagonally dominant matrices arise in many applications. With such a structure, it has been shown  recently that several linear algebra  problems can be solved much more accurately; see \cite{axy,axy2,ddy14a,ddy14b,ye08,ye09}.
In this section, we discuss related results on the $LDU$ factorization of diagonally dominant matrices, which will yield accurate solutions of linear systems and will be the key in our method to  compute a few smallest eigenvalues of extremely ill-conditioned matrices.

A key idea that makes more accurate algorithms possible is a representation (or re-parameterization) of diagonally dominant  matrices as follows.
\begin{definition}\label{def:matrix}
Given $M=(m_{ij})\in \R^{n\times n}$ with zero diagonals and
$v=(v_i)\in \R^n$, we use $\D(M, v)$ to denote the matrix
$A=(a_{ij})$ whose off-diagonal entries are the same as $M$
and whose $i$th diagonal entry is
$ v_i + \sum_{j\neq i}|m_{ij}| $, i.e.
\[
a_{ij} = m_{ij} \mbox{ for } i\ne j;
\mbox{ and } a_{ii} = v_i + \sum_{j\neq i}|m_{ij}|.
\]
\end{definition}

In this notation, for any matrix $A=[a_{ij}]$, let $A_D$ be the matrix
whose off-diagonal entries are the same as $A$ and whose diagonal
entries are zero and let $v_i =a_{ii} - \sum_{j\neq i}|a_{ij}| $
and $v=(v_1, v_2, \cdots,
v_n)^T$, then we can write $A=\D(A_D,v)$, which will be called
the representation of $A$ by the diagonally dominant parts $v$.
Through this equation, we use $(A_D,v)$ as the data (parameters) to define
the matrix (operator) $A$.
The difference between using all the entries of $A$ and using $(A_D, v)$
lies in the fact that, under small
entrywise perturbations, $(A_D,v)$ determines all entries of $A$ to the same
relative accuracy, but not vice versa.
Namely, $(A_D,v)$ contains more
information than the entries of $A$ do.

In general, a matrix $A=(a_{ij})$ is said to be diagonally dominant if
$|a_{ii}| \ge \sum_{j\neq i}|a_{ij}| $ for all $i$.
Throughout this work, we consider a diagonally dominant $A$ with nonnegative diagonals i.e.
$v_i = a_{ii} - \sum_{j\neq i}|a_{ij}|\ge 0 $ for all $i$. Diagonally dominant matrices with some negative diagonals can be scaled by a negative sign in the corresponding rows to turn into one with nonnegative diagonals. The need for scaling in such cases clearly does not pose any difficulties for the problem of solving linear systems.

In \cite{ye08}, we have developed a variation of the Gaussian elimination
to compute
the $LDU$ factorization of
an $n\times n$ diagonally dominant matrix represented as  $A=\D(A_D,v)$ with  $v\ge 0$ such that
the diagonal matrix $D$ has entrywise relative accuracy
in the order of machine precision while $L$  and $U$  are well-conditioned with normwise accuracy (see Theorem \ref{thm:thm2} below). The algorithm is based on the observation that the Gaussian elimination can be carried out on $A_D$ and $v$, and the entries of $v$ can be computed with no subtraction operation, generalizing our earlier algorithm for diagonally dominant M-matrices \cite{axy,axy2} and the GTH algorithm \cite{gth}. For completeness, we present the algorithm and its roundoff error properties below.

\begin{algorithm}\label{alg:lu} (\cite{ye08})
$LDU$ {\sc factorization of } $\D (A_D, v)$
\begin{tabbing}
123\=123\= 1236\= 12\= 3333\=3333\=  \kill
\>1 \> Input: $A_D=[a_{ij}]$ and  $v=[v_i]\ge 0$; \\
\>2 \> Initialize: $P=I$, $L=I$, $D=0$, $U=I$. \\
\>3 \> For $k=1:(n-1) $  \\
\>4 \>   \> For $i=k:n$ \\
\>5\>   \>  \>$a_{i i} =v_i + \sum_{j=k, j\ne i}^n | a_{i j} |$; \\
\>6 \>  \> End For\\
\>7 \>   \> If $\max_{i\ge k} a_{i i}=0$, stop; \\
\>8 \>   \> Choose a permutation $P_1$ for pivoting
            s.t. $A=P_1AP_1$ satisfies one of:\\
\>8a \>   \> \>a) if  diagonal pivoting: $a_{k k} = \max_{i\ge k} a_{i i}$;\\
\>8b \>   \> \>b) if   column diagonal dominance pivoting:
          $0\ne a_{k k} \ge \sum_{i=k+1}^n |a_{ik}|$;\\
\>9 \>   \>  $P=P_1P$; $L=P_1L P_1$; $U=P_1UP_1$; $d_k=a_{k k}$; \\
\>10 \>   \> For $i=(k+1):n$ \\
\>11 \>   \>  \>$l_{i k} =a_{i k} / a_{k k} $;
             $u_{k i}=a_{k i}/ a_{k k}$;  $a_{i k}=0$;\\
\>12 \>   \>  \>$v_i = v_{i} + |l_{i k}| v_{k}$; \\
\>13 \>    \>  \>For $j=(k+1):n$\\
\>14\>    \>  \>  \>$p= \sgn (a_{i j} - l_{i k} a_{k j} ) $; \\
\>15\>    \>  \>  \>$s= \sgn(a_{i j}) p $; \\
\>16\>    \>  \>  \>$t= - \sgn (l_{i k}  ) \sgn(a_{k j}) p $; \\
\>17\>    \>  \>  \>If $j = i$ \\
\>18\>    \>  \>  \>  \>$s=1$; $t=  \sgn (l_{i k}  ) \sgn(a_{k i}) $; \\
\>19\>    \>  \>  \>End if \\
\>20\>   \>  \> \>$v_i = v_{i} + (1-s)|a_{i j}|
                 + (1-t)|l_{i k}a_{k j} | $; \\
\>21\>    \>  \>  \>$a_{i j} =a_{i j} - l_{i k} a_{k j} $; \\
\>22\>    \>  \>End for\\
\>23 \>   \> End For\\
\>24\> End for \\
\>25\> $a_{n n} = v_n$; $d_n=a_{n n}$. \\ 
\end{tabbing}
\end{algorithm}

In output, we have $PAP^T=LDU$.
We have considered two possible pivoting strategies in line 8.
The column diagonal dominance pivoting ensures that $L$ is column diagonally dominant while $U$ is still row diagonally dominant. These theoretically guarantee that $L$ and $U$ are well-conditioned with
\begin{equation}\label{eq:condL}
\kappa_{\infty}(L):=\|L\|_{\infty}\|L^{-1}\|_{\infty} \le n^2 \;
\mbox{ and } \;
\kappa_{\infty}(U) \le 2 n;
\end{equation}
see \cite{pena04}. In practice, however, the diagonal pivoting (i.e. $a_{k k} = \max_{i\ge k} a_{i i}$ at line 8a) is usually sufficient to result in well-conditioned $L$ and $U$, but for the theoretical purpose, we assume that the column diagonal dominance pivoting will be used so that (\ref{eq:condL}) holds.

The following theorem characterizes  the accuracy achieved by Algorithm \ref{alg:lu}.

\begin{theorem}\label{thm:thm2}
Let $\widehat L=[\widehat{l}_{i k} ]$,
$\widehat D=diag\{ \widehat{d}_{ i} \}$ and $\widehat U=[\widehat{u}_{i k} ]$ be
the computed factors of $LDU$-factorization of $\D(A_D, v)$
by Algorithm \ref{alg:lu} and let $L=[{l}_{i k} ]$,
$D=diag\{ {d}_{ i} \}$ and $ U=[{u}_{i k} ]$ be the corresponding factors
computed exactly.
We have
\begin{eqnarray*}
\|\widehat{L} - L \|_\infty
& \le & \left( n \nu_{n-1} \bu +{\cal O} (\bu^2)\right) \|L \|_\infty ,\\
| \widehat{d}_i - d_i| & \le & \left( \xi_{n-1} \bu +{\cal O} (\bu^2) \right) d_i, \;\mbox{ for } 1\le i \le n,\\
\|\widehat{U} - U \|_\infty
& \le & \left( \nu_{n-1} \bu +{\cal O} (\bu^2)\right) \|U \|_\infty .
\end{eqnarray*}
where $\nu_{n-1} \le 14 n^3 $ and
$\xi_{n-1} \le 6 n^3.$
\end{theorem}

The above theorem was originally proved in \cite[Theorem 3]{ye08} with $\nu_{n-1} \le 6\cdot 8^{n-1} -2$ and
$\xi_{n-1} \le 5 \cdot 8^{n-1} -{5\over 2}$ but improved to the polynomial bound above in \cite[Theorem 4]{doko11}.
The bounds demonstrate that the computed $L$ and $U$ are normwise accurate and $D$ is entrywise accurate, regardless of the condition number of the matrix.
Since the permutation $P$ for pivoting does not involve any actual computations and roundoff errors, for the ease of presentation, we will assume from now on that $P=I$; that is the permutation has been applied to $A$.

The above accurate factorization was used in \cite{ye08} to accurately compute all singular values of $A$. Based on the Jacobi algorithm, the algorithm there is suitable for small matrices only. Here,  we will show that the accurate $LDU$ factorization can also be used to solve a linear system more accurately. Specifically, with the factorization $A=\widehat L \widehat D \widehat U$, we solve $Ax=b$ by the standard procedure:
\begin{equation}\label{eq:xhat}
\widehat L y =b; \; \; \widehat D z= y; \; \mbox{ and } \; \widehat U x = z,
\end{equation}
where the systems involving $\widehat L$ and $\widehat U$ are solved with forward and backward substitutions respectively, and the diagonal system is solved by $z_i=y_i/d_{ii}$.

The accuracy of the computed solution $\widehat x$ from (\ref{eq:xhat}) has been investigated by Dopico and Molera \cite{domo12} in the context of a rank revealing factorization. A factorization $A=XDY$ is said to be rank revealing if $X$ and $Y$ are well conditioned and $D$ is diagonal.
Let $\widehat X$, $\widehat D$, and $\widehat Y$ be the computed factors of a rank revealing factorization $ A=XDY$. We say it is an accurate rank revealing factorization of $A$ (see \cite{demmel99b}) if $\widehat X$ and $\widehat Y$ are normwise accurate and $\widehat D$ is entrywise accurate, i.e.,
\begin{equation}\label{eq:arrf}
\frac{\|\widehat X -X\|}{\|X\|} \le  \bu p(n); \;\;
\frac{\|\widehat Y -Y\|}{\|Y\|} \le  \bu p(n); \;\;
\;\; \mbox{ and } \;
|\widehat D-D|  \le  \bu p(n)|D|,
\end{equation}
where $p(n)$ is a polynomial in $n$. By (\ref{eq:condL}), the $LDU$ factorization defined by Algorithm \ref{alg:lu} is a rank revealing factorization. Furthermore, it follows from Theorem \ref{thm:thm2} that the computed factors by Algorithm \ref{alg:lu}  form an accurate rank revealing factorization. The following theorem describes the accuracy of the computed solution of $Ax=b$. In the theorem below, the norm $\|\cdot\|$ can be any matrix operator norm satisfying  $\|{\rm diag}\{d_i\}\|=\max_i |d_i|$.

\begin{theorem}\label{thm:dopico}
(\cite[Theorem 4.2]{domo12})
Let $\widehat X$, $\widehat D$, and $\widehat Y$ be the computed factors of a rank revealing factorization of $ A=XDY$  and assume that they satisfy (\ref{eq:arrf})
where  $p(n)$ is a polynomial of $n$ and $ X, D, $ and $Y$ are the corresponding exact factors. Assume also that the systems $X s = b$ and $Y x = w$ are solved with a backward
stable algorithm that when applied to any linear system $Bz = c$,  computes a
solution $\hat z $ that satisfies $(B+\Delta B) \hat z = c$; with $\|\Delta B\| \le \bu q(n) \|B\|$
where $q(n)$ is a modestly growing function of n such that $q(n) \ge 4 \sqrt{2} /(1-12\bu). $ Let
$g(n) := p(n)+q(n)+\bu p(n)q(n).$ Then, if $\hat x$ is the computed solution of $Ax = b$ through solving
\[
\widehat X y =b; \; \; \widehat D z= y; \; \mbox{ and } \; \widehat Z x = z,
\]
and if $\bu g(n)\kappa (Y)<1$ and $\bu g(n) (2+\bu g(n))\kappa(X) <1,$ then
\begin{eqnarray*}
 \frac{\|\widehat x -x\|}{\|x\|}  &\le&  \frac{\bu g(n) }{1-\bu g(n)\kappa (Y) } \left( \kappa (Y)+
   \frac{1+(2+\bu g(n)) \kappa(X) }{1-\bu g(n)(2+\bu g(n)) \kappa(X) }\frac{\|A^{-1}\| \|b\|}{\|x\|}  \right)\\
    &=&  \left( \bu g(n) + {\cal O}(\bu^2)\right) \max\{\kappa(X), \kappa(Y)\}  \frac{\|A^{-1}\| \|b\|}{\|x\|}
\end{eqnarray*}
\end{theorem}

Applying the above theorem  to $A=LDU$ with the infinity norm  and using the bounds on $\widehat L$, $\widehat D$ and $\widehat U$ in Theorem \ref{thm:thm2}, we have that the computed solution $\widehat x$ by (\ref{eq:xhat}) satisfies the bound above  with $p(n) =14 n^4$, $q(n) = n^2$. Hence,
further using (\ref{eq:condL}), we obtain
\begin{equation}\label{eq:xhatbound}
\frac{\|\widehat x -x\|_\infty}{\|x\|_\infty}  \le \left( \bu   n^2 g(n) + {\cal O}(\bu^2)\right)   \frac{\|A^{-1}\|_\infty \|b\|_\infty}{\|x\|_\infty}.
\end{equation}
where $ g(n) =14 n^6 +n^4$. We note that this  worst-case bound with a large coefficient $n^2 g(n)$ is derived for a dense matrix from combining the bounds for the $LDU$ factorization, for solving linear systems, and for the condition numbers, each of which is pessimistic. For a large sparse matrix that has ${\cal O} (n)$ nonzero entries, the coefficient can  be significantly reduced if we assume the $L$ and $U$ factors also have ${\cal O} (n)$ nonzero entries. In any case,  the bound can be expected to be too pessimistic to be useful for deriving a numerical bound in a practical setting.  The main interest of the bound is to demonstrate the independence of the error on any conditioning of the problem.

For the convenience of later uses, we rewrite (\ref{eq:xhatbound}) in the 2-norm as
\begin{equation}\label{eq:xhatbound2}
\|\widehat x -x\|\le {\cal O} ( \bu) \|A^{-1}\| \|b\|.
\end{equation}

\section{Smaller Eigenvalues of Extremely Ill-conditioned Matrices}
In this section, we discuss computations of a few smallest eigenvalues of an extremely ill-conditioned matrix $A$, namely we assume that $ \kappa (A)\bu$ is  of  order  1. For simplicity, we  consider a symmetric positive definite matrix $A$.
Many of the discussions are applicable to  nonsymmetric problems, although there may be additional difficulties associated with  sensitivity of individual eigenvalues caused by non-normality of the matrix.

Let   $0 < \lambda_1 \le \lambda_2 \le \cdots \le \lambda_n$ be the  eigenvalues of $A$ and suppose we are interested in computing the smallest eigenvalue $\lambda_1$.
%
For an  ill-conditioned matrix,
$\lambda_1$ is typically clustered; namely the relative spectral gap
\begin{equation}\label{eq:relgap}
\frac{\lambda_2 -\lambda_1}{\lambda_n - \lambda_2} = \frac{\lambda_2 -\lambda_1}{\lambda_1} \frac{1}{\lambda_n/\lambda_1-\lambda_2 / \lambda_1 } \approx \frac{\lambda_2 -\lambda_1}{\lambda_1} \frac{1}{ \kappa (A) }
\end{equation}
is  small unless $\lambda_2 \approx \lambda_n$.
Therefore, although one can  compute $\lambda_1$ using the Lanczos algorithm,   its speed of convergence is determined by the relative spectral gap, and then a direct application  to $A$ will result in slow or no convergence. An efficient way to deal with the clustering is to apply the Lanczos method  to the inverse $A^{-1}$, which has a much better relative spectral gap
$\frac{\lambda_1^{-1} -\lambda_2^{-1}}{\lambda_2^{-1} - \lambda_n^{-1}} >  \frac{\lambda_2 -\lambda_1}{\lambda_1} $.

Consider using the Lanczos method on  $A^{-1}$ or simply  inverse iteration and compute its largest eigenvalue $\mu_1=\lambda_1^{-1}$.  One may argue that, by (\ref{eq:relerror}), since $\mu_1$ is the largest eigenvalue, it can be computed accurately, from which $\lambda_1=\mu_1^{-1}$ can be recovered accurately. However, we can compute $\mu_1$ accurately only if $A^{-1}$ is explicitly available or the matrix-vector multiplication $A^{-1}v$ can be accurately computed, which is not the case if $A$ is ill-conditioned. Specifically, if $u =A^{-1}v$ is computed by factorizing $A$ and solving $A u =v$, then the computed solution $\widehat u$ is backward stable. That is  $(A+\Delta A) \widehat u = v$ for some $\Delta A$ satisfying $\|\Delta A\| \le {\cal O} (\bu) \|A\|$. Then
\begin{equation}\label{eq:errorvec}
\widehat u - u = -A^{-1} \Delta A \widehat u
 \; \mbox{ and } \;
\|\widehat u - u\| \le {\cal O} (\bu)\|A^{-1}\| \| A \| \|\widehat u\|.
\end{equation}
So, if $\bu \|A^{-1}\| \| A \| $ is of order 1, very little accuracy can be expected of the computed result  $\widehat u$. Then little accuracy can be expected of the computed eigenvalue as  its accuracy is clearly limited by that of the matrix operator $A^{-1}v$. We further illustrate this point by looking at how an approximate eigenvalue is computed.

Note that all iterative methods is based on constructing an approximate eigenvector $x_1$, from which an approximate eigenvalue is computed, essentially as the Rayleigh quotient\footnote{In the Lanczos algorithm, the Ritz value is computed as the eigenvalue of the projection matrix, which is still the Rayleigh quotient of the Ritz vector.} $\rho_1 := \frac{x_1^T A^{-1}x_1}{x_1^T x_1}$. Unfortunately, for extremely ill-conditioned $A$, this Rayleigh quotient can not be computed with much accuracy. Specifically, to compute $\rho_1$, we first need to compute $A^{-1}x_1$. If $\widehat u_1$ is the computed solution to $Au_1 = x_1$, the computed Rayleigh quotient satisfies
\begin{equation}\label{eq:hatrho1}
\hat \rho_1 = fl\left( \frac{x_1^T \widehat u_1}{x_1^T x_1}\right)= \frac{x_1^T \widehat u_1}{x_1^T x_1} + e
 \; \mbox{ with  } \;
 |e| \le {\cal O} (\bu)  \frac{|x_1|^T | \widehat u_1|}{x_1^T x_1} \le {\cal O} (\bu) \frac{\|\widehat u_1\|}{\|x_1\|}.
\end{equation}
Now, it follows from (\ref{eq:errorvec}) that
\[
|\hat \rho_1 - \rho_1 |=  \left|\frac{x_1^T (\widehat u_1-u_1) }{x_1^T x_1} +e \right|
\le {\cal O} (\bu) \kappa_2 (A) \frac{\|\widehat u_1\|}{\|x_1\|}.
\]
Using (\ref{eq:errorvec}) again and $\|u_1\| =\|A^{-1}x_1\| \le \mu_1 \|x_1\|$, we have
\[
\frac{\|\widehat u_1\|}{\|x_1\|} \le  \frac{\| u_1\|}{\|x_1\|} + \frac{\|\widehat u_1-u_1\|}{\|x_1\|} \le\mu_1+{\cal O} (\bu)\kappa_2 (A) \frac{\|\widehat u_1\|}{\|x_1\|}
\]
from which it follows that
\[
\frac{\|\widehat u_1\|}{\|x_1\|} \le  \frac{\mu_1}{1-{\cal O} (\bu)\kappa_2 (A) }.
\]
Therefore,
\[
|\hat \rho_1 - \rho_1 |\le  \frac{{\cal O} (\bu)\kappa_2 (A)}{1-{\cal O} (\bu)\kappa_2 (A) }  \mu_1 = {\cal O} (\bu)\kappa_2 (A) \mu_1.
\]
With $\rho_1 \le \mu_1$, the relative error of the computed Rayleigh quotient $\hat \rho_1$ is expected to be of order $ {\cal O} (\bu)\kappa_2 (A)$.
Note that this is independent of the algorithm that we use to obtain an approximate eigenvector $x_1$. Indeed, this is the case even when $x_1$ is an exact eigenvector. That is, even if we have the exact eigenvector, we are still not able to compute a corresponding eigenvalue with any accuracy if it is computed numerically through the  Rayleigh quotient.

We point out that the same difficulty occurs if we compute approximate eigenvalue $\lambda_1$ directly using the Rayleigh quotient of $A$, i.e. using $\rho_1 = \frac{x_1^T Ax_1}{x_1^T x_1}$. In this case, as in (\ref{eq:hatrho1}), the computed Rayleigh quotient $\hat \rho_1$ satisfies
\[
\hat \rho_1 = fl\left( \frac{x_1^T \widehat y_1}{x_1^T x_1}\right)= \frac{x_1^T \widehat y_1}{x_1^T x_1} + e
 \; \mbox{ with  } \;
 |e| \le {\cal O} (\bu)  \frac{|x_1|^T | \widehat y_1|}{x_1^T x_1} \le {\cal O} (\bu) \frac{\|\widehat y_1\|}{\|x_1\|}
\]
where $\widehat y_1 = fl(Ax_1)$. Set $ y_1 =A x_1$.  Then $\widehat y_1 = fl(Ax_1) =y_1 +f$ with $|f| \le \frac{n \bu}{1-n \bu} |A| |x_1|$.  Hence, $\|\widehat y_1\| \le (1+{\cal O} (\bu))\|A\| \|x_1\|$ and $\|\widehat y_1-y_1\| \le {\cal O} (\bu)\|A\| \|x_1\|$. It follows that $|e| \le {\cal O} (\bu) \|A\|$ and, for $\rho_1 \approx \lambda_1$,
\begin{equation}\label{eq:rqiA}
\frac{|\hat \rho_1 - \rho_1 |}{\rho_1} = \frac{1}{\rho_1} \left|\frac{x_1^T (\widehat y_1-y_1) }{x_1^T x_1} +e \right|
\le  \frac{{\cal O} (\bu) \|A\|}{\rho_1} \approx {\cal O} (\bu) \|A^{-1}\| \| A \|.
\end{equation}
Again, the accuracy of $\hat \rho_1 $ is limited by the condition number. This also shows that, if we apply the Lanczos method directly to $A$, then the smallest Ritz eigenvalue  can not be accurately computed, even if the process converges.

The above discussions demonstrate why all existing methods will have difficulties computing  the smallest eigenvalue accurately. It also highlights that the roundoff errors encountered in computing $A^{-1}v$  are the root cause of the problem. This also readily suggests a remedy: compute $A^{-1}v$ more accurately. It turns out that a more accurate solution to $Au=v$, in the sense of (\ref{eq:xhatbound2}), is sufficient, and this can be done for  a diagonally dominant matrix or  a product of diagonally dominant matrices.

Let $A=\D (A_D, v)$ be a (possibly nonsymmetric) diagonally dominant matrix with $v \ge 0$. We apply the Lanczos method (or the Arnoldi method in the nonsymmetric case) to $A^{-1}$  or simply using   inverse iteration. Then at each iteration, we need to compute $u =A^{-1} v$ for some $v$. This is done by computing the accurate $LDU$ factorization of $A$ by Algorithm \ref{alg:lu} and then solve $Au=v$ by (\ref{eq:xhat}). The computed solution $\widehat u$ then satisfies (\ref{eq:xhatbound}), i.e.
\begin{equation}\label{eq:ubound}
\|\widehat u -u\| \le {\cal O} (\bu) \|A^{-1}\| \|v\|.
\end{equation}
This error depends on   $ \|A^{-1}\|$ but is as accurate as multiplying  the exact $A^{-1}$ with $v$.
Specifically, if we have $A^{-1}$ exactly, say $A^{-1}= B$, and we compute $A^{-1}v$ by computing $Bv$ in a floating point arithmetic, the error is bounded as
\[
\|fl(Bv)- Bv \| \le  \frac{n^2 \bu}{1-n\bu} \|B\| \|v\|=\frac{n^2 \bu}{1-n\bu} \|A^{-1}\| \|v\|.
\]
So, the error of $\widehat u$ is of the same order as that of $fl (Bv)$.

Therefore, applying the Lanczos algorithm to $A^{-1}$ by solving $Au=v$ using Algorithm \ref{alg:lu} and  (\ref{eq:xhat}) has essentially the same roundoff effects as applying it to $B$ in a floating point arithmetic.
In that case,  $\mu_1 :=\lambda_1^{-1}$, being the largest eigenvalue of  $ B=A^{-1}$, is computed accurately (see (\ref{eq:relerror})).
In light of the earlier discussions  with respect to the Rayleigh quotient, we also show that the Rayleigh quotient $\rho_1 := \frac{x_1^T A^{-1}x_1}{x_1^T x_1}$ for an approximate eigenvector $x_1$  can be computed accurately. Specifically, if $\widehat u_1$ is the computed solution to $Au_1 = x_1$, then $\|\widehat u_1 -u_1\| \le {\cal O} (\bu) \|A^{-1}\| \|x_1\|$ by (\ref{eq:ubound}) and the computed Rayleigh quotient $\hat \rho_1 := fl\left( \frac{x_1^T \widehat u_1}{x_1^T x_1}\right)$ satisfies (\ref{eq:hatrho1}). Thus, the error is bounded as
\[
|\hat \rho_1 - \rho_1 |=  \left|\frac{x_1^T (\widehat u_1-u_1) }{x_1^T x_1} +e \right|
\le {\cal O} (\bu) \|A^{-1}\| ={\cal O} (\bu) \mu_1,
\]
where we have used
\[
|e| \le {\cal O} (\bu) \frac{\|\widehat u_1\|}{\|x_1\|} \le {\cal O} (\bu) \frac{\|\widehat u_1-u_1\|}{\|x_1\|}+
{\cal O} (\bu) \frac{\| u_1\|}{\|x_1\|} \le {\cal O} (\bu) \|A^{-1}\|.
\]
With $\rho_1\approx \mu_1$, thus,  $\rho_1$ is computed  with a relative error of order $\bu$.

%

Now, suppose a matrix $A$ is not diagonally dominant but can be written as a product of two or more diagonally dominant matrices, say $A=A_1 A_2$. Then we can factorize $A_1$ and $A_2$ using Algorithm \ref{alg:lu} and then solve $Au=v$ through $A_1 w= v$ and then $A_2 u  =w$ as in (\ref{eq:xhat}). Let $\widehat w$  and $\widehat u$ be the computed solutions to $A_1 w= v$ and   $A_2 u  =\widehat w$ respectively. Using (\ref{eq:xhatbound2}), they satisfy
\[
\|\widehat w -w\| \le {\cal O} (\bu) \|A_1^{-1}\| \|v\|
\]
and
\begin{eqnarray*}
\|\widehat u -u\| &\le &{\cal O} (\bu) \|A_2^{-1}\| \|\widehat w\| \\
&\le& {\cal O} (\bu) \|A_2^{-1}\| \| w\| + {\cal O} (\bu) \|A_2^{-1}\| \|\widehat w-w\| \\
&\le & {\cal O} (\bu) \|A_2^{-1}\| \|A_1^{-1}\| \|v\|.
\end{eqnarray*}
Thus, writing $\gamma =  \|A_1^{-1}\| \|A_2^{-1}\|/ \|A^{-1}\| $, we have
\begin{equation}\label{eq:uw}
\|\widehat u -u\|\le {\cal O} (\bu) \gamma \|A^{-1}\| \|v\|.
\end{equation}
So, as discussed above, the largest eigenvalue of $A^{-1}$ can be computed to an accuracy of order ${\cal O} (\bu) \gamma$, which should be satisfactory as long as $\gamma$ is not an extremely large constant.

In the next section, we present examples from differential operators to demonstrate that the approaches discussed in this section result in accurate computed eigenvalues.

\section{Eigenvalues of Differential Operators}

The eigenvalue problem for self-adjoint differential operator arises in
dynamical analysis of vibrating elastic structures such as
bars, beams, strings, membranes, plates,
and solids. It takes the form of ${\cal L} u = \lambda u$ in
${\cal O}mega$ with a suitable homogeneous boundary condition, where ${\cal L}$
is a differential operator and ${\cal O}mega$ is a bounded domain in
one, two, or three dimensions. For example, the natural vibrating
frequencies of a thin membrane stretched over a bounded domain ${\cal O}mega$
and fixed along the boundary $\partial {\cal O}mega$ is determined by
the eigenvalue problem for a second order differential operator \cite[p.653]{babu}
\begin{equation}\label{eq:2order}
- \nabla ({\bf c} (x,y) \nabla u (x,y)) = \lambda u (x,y) \;\;
\mbox{ in }
(x,y) \in
{\cal O}mega ,
\end{equation}
with a homogeneous boundary condition
$ \alpha u + \beta {\partial u \over \partial n } = 0$ for
$(x,y) \in \partial {\cal O}mega$.
On the other hand, a vibrating plate on ${\cal O}mega$ 
is described by
one for a fourth order differential operator
\cite[p.16]{wein}
\begin{equation}\label{eq:4order}
\Delta^2 u (x,y) = \lambda u (x,y) \;\;
\mbox{ in }
(x,y) \in
{\cal O}mega,
\end{equation}
with the natural boundary condition $ u = \Delta u = 0$ (if the plate is simply supported at the boundary) or the Dirichlet boundary condition
$ u = {\partial u \over \partial n } = 0$  (if the plate is clamped at the boundary) for
$(x,y) \in \partial {\cal O}mega$.
The spectrum of such
a differential operator consists of an infinite sequence of eigenvalues with
an accumulation point at the infinity:
\begin{equation}\label{eq:seq}
\lambda_1 \le \lambda_2 \le \cdots \cdots \uparrow +\infty .
\end{equation}
When computing eigenvalues for practical problems,
it is usually lower frequencies, i.e. the eigenvalues at the left end of the spectrum,
such as $\lambda_1$, that are of interest.

Numerical solutions of the differential eigenvalue problems
 typically involve first discretizing the differential equations
and then solving a matrix eigenvalue problem. Then, it is usually a few smallest eigenvalues of the matrix that
well approximate the eigenvalues of the
differential operators and are of interest.
The finite difference discretization
leads to a standard eigenvalue problem $Ax=\lambda x$, while the
finite element methods or the spectral methods
result in a generalized eigenvalue problem
$Ax=\lambda Bx$. Although our discussions may be relevant to all
three discretization methods,
we shall consider the finite difference method here as they often give rise to the diagonal dominant structure that can  be utilized in this work.

Let $A_h$ be an $n\times n$ symmetric finite difference discretization of a self-adjoint positive definite
differential
operator ${\cal L} $ with $h$ being the meshsize and let
\begin{equation}\label{eq:seqh}
0< \lambda_{1, h} \le \lambda_{2, h} \le \cdots \le \lambda_{n,h}
\end{equation}
be its eigenvalues.
Under some conditions on the operator and the domain \cite{kell},
we have for a fixed $i$, $\lambda_{i, h}\rightarrow \lambda_i$
as $h\rightarrow 0$,
while the largest eigenvalue $\lambda_{n,h}$ approaches $\infty$.
 Then, the condition number $\kappa_2 (A_h)$ of $A_h$
approaches the infinity as the meshsize decreases to $0$.
As discussed in the introduction (see (\ref{eq:relerror})),
if $\widehat \lambda_{1,h}$ is the computed smallest eigenvalue, its relative
error is expected to be of order $\kappa_2 (A_h) \bu$.
Hence, we have a situation that, as the meshsize decreases, the
discretization error $\lambda_{1,h}-\lambda_{1}$
decreases, but our ability to  accurately  compute $\lambda_{1,h}$  also
decreases. Indeed, the overall error
\[
|\widehat \lambda_{1,h}-\lambda_{1}| \le
|\widehat \lambda_{1,h} -\lambda_{1,h}|+|\lambda_{1,h}-\lambda_{1}|
\le \lambda_{1,h} \kappa_2 (A_h) {\cal O} (\bu) +|\lambda_{1,h}-\lambda_{1}|.
\]
will at some point be dominated by the computation  error $|\widehat \lambda_{1,h} -\lambda_{1,h}|$.
We illustrate with an example.

\smallskip

\noindent {\bf  Example 1:}
Consider the biharmonic eigenvalue problem $\Delta^2  u=\lambda u$ on
the unit square $[0, 1]^2$. With the
simply supported boundary condition $u= {\partial^2 u \over \partial n^2} =0$,
the eigenvalues are known explicitly  \cite{bare} and the smallest one
is $\lambda_1 =4 \pi^4$.
The 13-point finite difference discretization with a meshsize
$h = 1/(N+1)$ leads to a block five-diagonal
matrix (see \cite{bare} or \cite[p.131]{migr}). It is easy
to check that the discretization matrix can be written as
$A_h = \frac{1}{h^4} T_{N\times N}^2 $, where $T_{N\times N} =T_N  \kron I   + I \kron T_N$,
$T_N$ is the $N\times N$ tridiagonal matrix
with diagonals being $2$ and off-diagonals being $-1$, and $\kron $
is the Kronecker product. The   eigenvalues of $A_h$
are also known exactly and $\lambda_{1,h} =\frac{1}{h^4} 64 \sin^4 ( \pi h/2)$.

We compute the smallest eigenvalue of $A_h$ using the implicitly restarted
Lanczos algorithm
with shift-and-invert (i.e. using {\tt eigs(A,1,$\sigma$)} of \MATLAB with
$\sigma=0$ with its default termination criterion) and let $\widehat \lambda_{1,h}$ be the computed result. We
compare the discretization error $|\lambda_{1,h}-\lambda_{1}|$,
the computation error $|\widehat \lambda_{1,h} -\lambda_{1,h}|$
and the overall approximation error
$|\widehat \lambda_{1,h}-\lambda_{1}|$. We use increasingly fine mesh
(with $N=2^{k}-1$ for $k = 1,2,\cdots$)
and plot the corresponding errors
against $k$ in Figure \ref{fig:e1} (left). We test the value of $k$ up to $k=8$ for the 2-D problem. To test larger values of $k$, we use the corresponding 1-D biharmonic problem
where $A_h=\frac{1}{h^4}T_N^2$, $\lambda_{1,h} =\frac{16}{h^4}  \sin^4 ( \pi h/2)$ and $\lambda_1=\pi^4$.
Here we use $k$ up to $k=15$ and   plot the result in Figure \ref{fig:e1} (right).

\begin{figure}[th]
\caption{\small { Example 1.} Left: 2-D problem;
Right: 1-D problem.
solid line: computation error;
dashed line: discretization error; +-line: overall error.
\label{fig:nsfex1}
}
\centerline{
\epsfig{figure=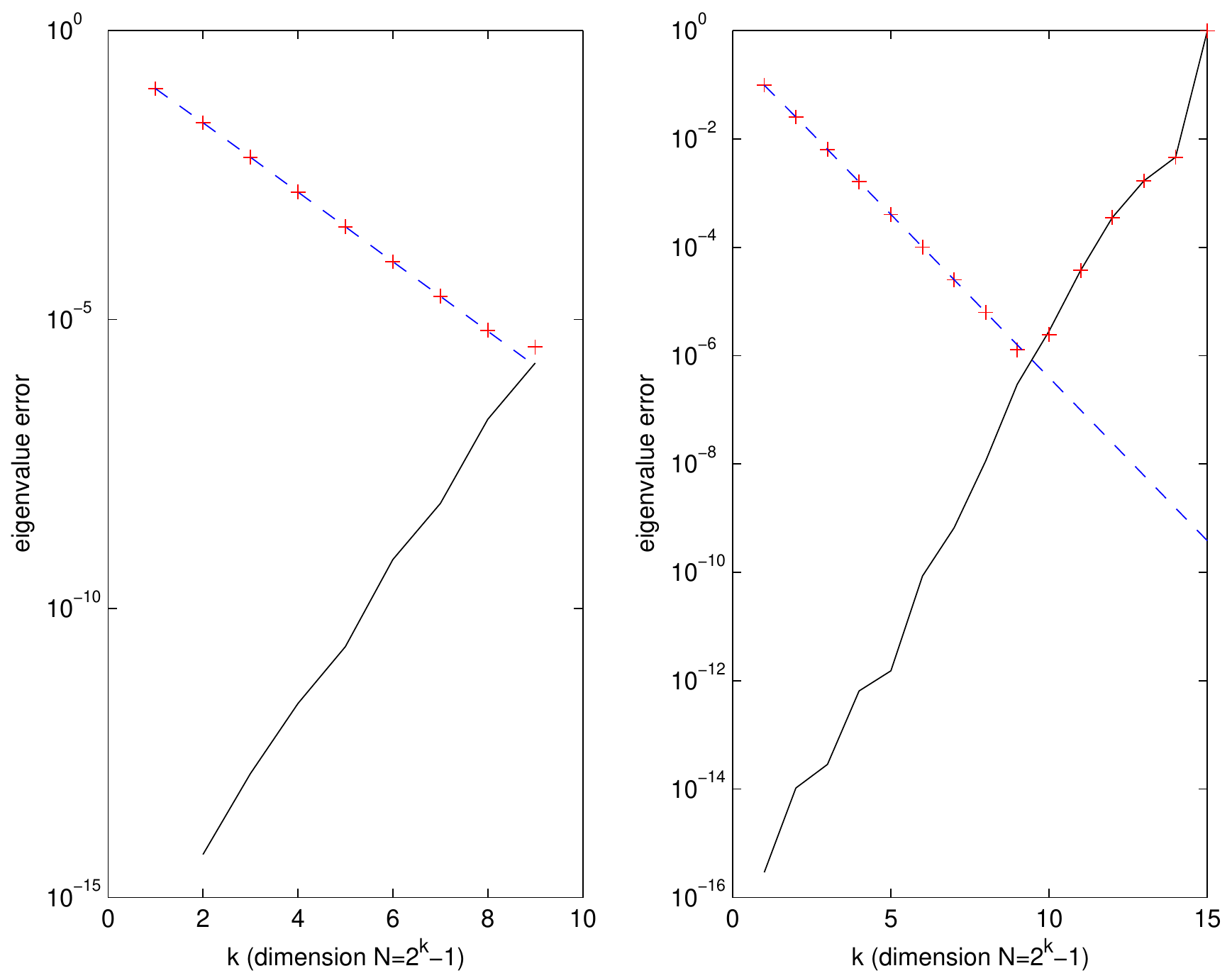,width=4.5in,height=3in}
}
\label{fig:e1}
\end{figure}


It is clear that as the meshsize decreases, the discretization error (dashed line)
decreases, while the computation error (solid line) increases.
At some point
($k=9$ in this case), the computation error dominates the discretization error
and the total error ($+$-line) increases from that point on  to ${\cal O} (1)$. So,
at some point, using
finer discretizations actually increases the final   error.

We note that the large error in the computed eigenvalue is not due to possible numerical complications of the Lanczos method we used to compute the eigenvalues of $A_h$. Indeed, other than the implicitly restarted Lanczos method ({\tt eigs}), we have also used simply   inverse iteration and the Lanczos algorithm with full reorthogonalization as combined with the Cholesky factorization for inverting $A$ to compute the smallest eigenvalue and obtained similar results.

We also note that for this particular problem, the eigenvalues can be computed as the square of the eigenvalues of $T_{N\times N}$ (or $T_N$ in the 1-dimensional case). This is not the case for a general operator involving weight functions or a different boundary condition. Here, we compute the eigenvalues directly from $A_h$ to illustrate the difficulty encountered by ill-conditioning.
%

\subsection{Second Order Differential Operators.}
A standard finite difference
discretization of second order differential operators such as (\ref{eq:2order}) is  the five-point scheme, which typically has a condition number of order ${\cal O} (h^{-2})$. Then
the relative error for
the smallest eigenvalue computed
is expected to be ${\cal O} (h^{-2}) \bu$. For modestly small
$h$, this relative error may be  smaller than the
discretization error and would not cause any problem.
However, there are situations where $\lambda_1$ may be 0 or very close to 0,
resulting in a condition number significantly greater than $h^{-2}$ and hence low accuracy
of the computed eigenvalue even when $h$ is modestly small.
This is the case if we have a periodic boundary condition or Neumann boundary condition. Then the smallest
eigenvalue may be computed with little accuracy (see Example 2 below). Note that for a zero eigenvalue, we need to consider the absolute error, which will be proportional to ${\cal O} (h^{-2}) \bu$.

For the second order differential operators, fortunately, the finite difference discretization matrices
are typically diagonally dominant; see \cite[p.211]{varga}. Then, a few smallest eigenvalues can be computed accurately by applying  standard iterative methods to $A^{-1}$ as discussed in \S 3. We demonstrate this in the following numerical example.

In this and Examples 3 and 4 later, we have used both   the Lanczos algorithm with full
reorthogonalization for $A^{-1}$ and   inverse iteration and found the results to be similar. With the smallest eigenvalue well separated, both methods converge quickly but the Lanczos algorithm (without restart) typically requires fewer iterations. On the other hand,  inverse iteration sometimes improves the   residuals  by up to one order of magnitude.
We report the results obtained by  inverse iteration only.

\smallskip

\noindent {\bf Example 2:}
Consider the eigenvalue problem for ${\cal L} u =
-\Delta u + \rho(x,y) u$ on
the 2-d unit square $[0, 1]^2$ with the
periodic boundary condition $u(0,y)=u(1,y)$ and $u(x,0)=u(x,1)$. If $\rho (x,y)=0$,
the smallest eigenvalue of $\LL$ is $0$. If $\rho (x,y)\le \epsilon$ for some small $\epsilon$, then the smallest eigenvalue is at most $\epsilon$. In particular,
we test the case that $\rho(x,y)=\rho$ is a small constant in this example. Then
the smallest eigenvalue of ${\cal L}$ is exactly
$\rho$
 with $u=1$ being a corresponding eigenfunction.

The 5-point finite difference discretization with a meshsize
$h = 1/(N+1)$ leads to
$A_h = \frac{1}{h^2} (\widehat T_{N+1} \kron I + I \kron \widehat T_{N+1} ) + \rho I $
where  $\widehat T_{N+1}$ is  the $(N+1)\times (N+1)$  matrix that is the same as $T_{N+1}$ in Example 1 (i.e tridiagonal with 2 on the diagonal and $-1$ on the subdiagonals) except at the $(1,N+1)$ and $(N+1,1)$ entries which
 are $-1$.  
The smallest eigenvalue of $A_h$ is still $\rho$ while the largest is of order $h^{-2}$.

We compute $\lambda_{1,h}$ by applying   inverse iteration
with the inverse operator $A_h^{-1}x_k$ computed
 using the Cholsky factorization or  using
the accurate $LDL^T$ factorization (Algorithm \ref{alg:lu}) of $A_h$.
Denote the computed eigenvalues
by $\mu_1^{chol}$ and $\mu_1^{aldu} $ respectively
in Table \ref{tbl:ex1}, where the test is run with $h=2^{-k}$ with $k=3,4,\cdots,9$ and $\rho=1e-8$.
The termination criterion is
$\|A^{-1}x_k-\mu_k x_k\|/  \|x_k\|  \le n \bu |\mu_k|$.

For this problem, $\rho$ is always the smallest eigenvalue of $A_h$. As $h$ decreases, the condition number of $A_h$ increases, which cause the relative error of the computed eigenvalue  $\mu_1^{chol}$ to increase steadily. On the other hand,  $\mu_1^{aldu} $ computed by the new approach using the accurate factorization  has a relative error in the order of machine precision $ \bu$ independent of $h$.

\begin{table}[hbt] \small
\begin{center}
\caption{\small {Example 2}: approximation of
$\lambda_1 =1e-8$ (
$\mu_1^{chol}$ - computed eigenvalue by Cholesky;
$\mu_1^{aldu}$ - computed eigenvalue by Algorithm \ref{alg:lu}.)
}
\begin{tabular}{c|c|c|c|c}
$h$ & 
$\mu_1^{chol}$ & $\frac{|\lambda_{1}-\mu_1^{chol}|}{\lambda_{1}}$
& $\mu_1^{aldu}$ & $\frac{|\lambda_{1}-\mu_1^{aldu}|}{\lambda_{1}}$ \\
\hline
 1.3e-1  	 & 	  1.00000132998311550e-8  	 & 	  1.3e-6  	 & 	 9.99999999999999860e-9  	 & 	  1.7e-16\\
 6.3e-2  	 & 	  9.99988941000340750e-9  	 & 	  1.1e-5  	 & 	 1.00000000000000020e-8  	 & 	  1.7e-16\\
 3.1e-2  	 & 	  9.99983035717776210e-9  	 & 	  1.7e-5  	 & 	 1.00000000000000020e-8  	 & 	  1.7e-16\\
 1.6e-2  	 & 	  1.00008828596534290e-8  	 & 	  8.8e-5  	 & 	 1.00000000000000000e-8  	 & 	  0\\
 7.8e-3  	 & 	  9.99709923055774120e-9  	 & 	  2.9e-4  	 & 	 1.00000000000000000e-8  	 & 	  0\\
 3.9e-3  	 & 	  1.00122462366335050e-8  	 & 	  1.2e-3  	 & 	 1.00000000000000050e-8  	 & 	  5.0e-16\\
 2.0e-3  	 & 	  1.00148903549009850e-8  	 & 	  1.5e-3  	 & 	 1.00000000000000040e-8  	 & 	  3.3e-16
\end{tabular}

\label{tbl:ex1}
\end{center}
\end{table}

\subsection{Biharmonic operator with the natural boundary condition}
A thin vibrating elastic  plate on ${\cal O}mega$ that is simply supported at the boundary $\partial {\cal O}mega$ is described by
the eigenvalue problem for the biharmonic operator (\ref{eq:4order}) with the natural boundary condition
\begin{equation}\label{eq:support}
u= \Delta u =0 \;\; \mbox{ for }
(x,y) \in
\partial {\cal O}mega .
\end{equation}

Discretization of such operators are not diagonally dominant. However,
the biharmonic operator   $\Delta^2 $ has a  form as a composition of two second order differential operators.
%
%
Indeed, with the
simply supported boundary condition (\ref{eq:support}), the eigenvalue equation has the product form
\[
\left\{
  \begin{array}{ll}
    -\Delta u = v, & \hbox{with} \; u=0 \mbox{ for }
(x,y) \in
\partial {\cal O}mega \\
    -\Delta v  = \lambda u & \hbox{with}\; v=0 \mbox{ for }
(x,y) \in
\partial {\cal O}mega
  \end{array}
\right.\]
Then, the two Laplacian operators with the same boundary condition will have the same discretization matrix $A_h$. Thus the final discretization has the product form $B_h = A_h^2$. Since $A_h$ is typically diagonally dominant, $B_h$ is a product of two diagonally dominant matrices, and then as discussed in Section 3, its eigenvalue can also be computed accurately by using accurate factorizations of $A_h$. Note that, with respect to (\ref{eq:uw}), we have $\gamma = \|A_h\|\|A_h\|/\|B_h\|=1$.

For this problem, the eigenvalue can also be obtained by computing the eigenvalue of $A_h$ first and then squaring. With $A_h$ having more modest ill-conditioning, its eigenvalues can be computed more accurately. However, this approach does not generalize to  problems such as stretched plates with a nonuniform stretch factor $\rho(x,y)$  as described by $\Delta^2 u (x,y) - \rho \Delta  u (x,y)= \lambda u (x,y)$. 
On the other hand, its discretization can still be expressed in a product form; that is $B_h =  ( A_h+ D)A_h $ where $D={\rm diag}\{\rho(x_i,y_j)\}$. Then if $\rho(x, y) \ge 0$, $A_h+D$ is also diagonally dominant. Thus the method described in \S 3 can compute  a few smallest eigenvalues of $B_h$ accurately.
We illustrate this in the following numerical example. We use the 1-dimensional problem so as to test some really  fine meshes.

%
%

\smallskip

\noindent {\bf Example 3:}
Consider computing the smallest eigenvalue of the eigenvalue problem:
$ {d^4 v \over dx^4} - \rho(x) {d^2 v \over dx^2} =\lambda v $ on $[0, 1]$
with the natural boundary condition $v(0)=v''(0)=v(1)=v''(1)=0$. This models vibration of beams; see \cite[p.15]{wein}. Discretization on a uniform mesh with $h=1/(N+1)$ leads to
\begin{equation}\label{eq:fact}
B_h u_N = \lambda h^4 u_N ; \;\; \mbox{ where }\;\; B_h =  T_{N}^2+h^2 D T_{N} =  (T_{N} +h^2 D )T_{N}
\end{equation}
where $D={\rm diag}\{\rho(ih)\}$.
We test the case with a constant stretch factor $\rho=1$ so that $D=I$ and the eigenvalues of the differential operator  are exactly known. Namely, $\lambda_1 = \pi^4 +\pi^2 $.

We compute $\lambda_{1,h}$ by applying  inverse iteration\footnote{As mentioned before,
similar results are obtained by applying the Lanczos algorithm with full reorthogonalization to $B_h^{-1}$.}
to $B_h$ and we use the termination criterion $\|B_h^{-1}x_k-\mu_k x_k\|/  \|x_k\|  \le n \bu |\mu_k|$.
In applying $B_h^{-1}$, we compare the methods of  using
 the Cholesky factorization of $B_h$ and   using
the accurate $LDL^T$ factorization of  $T_{N}+h^2 D$ and $ T_{N}$ in (\ref{eq:fact}).
We denote the computed eigenvalues by $\mu_1^{chol}$ and $\mu_1^{aldu} $ respectively and present the results
in Table \ref{tbl:ex2}.

\begin{table}[hbt] \small
\begin{center}
\caption{\small {Example 3}: approximation of
$\lambda_1 =\pi^4 +\pi^2$ (
$\mu_1^{chol}$ - computed eigenvalue by Cholesky;
$\mu_1^{aldu}$ - computed eigenvalue by Algorithm \ref{alg:lu}.)
}
\begin{tabular}{c|c|c|c|c}
$h$ &
$\mu_1^{chol}$ & $\frac{|\lambda_{1}-\mu_1^{chol}|}{\lambda_{1}}$
& $\mu_1^{aldu}$ & $\frac{|\lambda_{1}-\mu_1^{aldu}|}{\lambda_{1}}$ \\
\hline
 7.8e-3      &    1.07268420705487220e2      &    9.6e-5     &   1.07268420682006790e2       &    9.6e-5\\
 3.9e-3      &    1.07276126605835940e2      &    2.4e-5     &   1.07276126662112130e2       &    2.4e-5\\
 2.0e-3      &    1.07278050199309870e2      &    6.0e-6     &   1.07278053236552470e2       &    6.0e-6\\
 9.8e-4      &    1.07278541687712870e2      &    1.4e-6     &   1.07278534885126120e2       &    1.5e-6\\
 4.9e-4      &    1.07278608040777410e2      &    8.1e-7     &   1.07278655297579520e2       &    3.7e-7\\
 2.4e-4      &    1.07276683387428220e2      &    1.9e-5     &   1.07278685400712600e2       &    9.4e-8\\
 1.2e-4      &    1.07255355496649590e2      &    2.2e-4     &   1.07278692926496390e2       &    2.3e-8\\
 6.1e-5      &    1.10218434482696150e2      &    2.7e-2     &   1.07278694807942740e2       &    5.8e-9\\
 3.1e-5      &    1.16309014323327090e2      &    8.4e-2     &   1.07278695278302880e2       &    1.5e-9\\
 1.5e-5      &    2.50710434033222100e2      &    1.3e0      &   1.07278695395894270e2       &    3.7e-10

\end{tabular}

\label{tbl:ex2}
\end{center}
\end{table}

From the table, the error of $\mu_1^{chol}$ decreases quadratically  for $h$ up to $h\approx 5e-4$, at which point further decrease of $h$ actually increases the error. When $h\approx 1e-5$, there is no accuracy left in the computed result $\mu_1^{chol}$. On the other hand,  the error for  $\mu_1^{aldu}$ continues to decrease quadratically  and maintains an accuracy  to the order of machine precision.

\subsection{Biharmonic operator with the Dirichlet boundary condition}
A thin vibrating   plate on ${\cal O}mega$ that is clamped at the boundary $\partial {\cal O}mega$ is described by
the eigenvalue problem for the biharmonic operator (\ref{eq:4order}) with the Dirichlet boundary condition
\begin{equation}\label{eq:clamp}
u= {\partial u \over \partial n } =0 \;\; \mbox{ on }
(x,y) \in
\partial {\cal O}mega .
\end{equation}
With this boundary condition,
the standard 13-point discretization does not  have   a product form; see \cite{ehrl71}. For the   1-dimensional case, it turns out that we can derive a suitable discretization at the boundary so that the resulting matrix is a product of two  diagonally dominant matrices. Here we present the details of this scheme.

We consider one dimensional fourth order bi-harmonic operator
\begin{equation}\label{bihar1d}
{d^4 v \over dx^4}  =\lambda v \;\;\mbox{ on }\;\;[0, 1]
\end{equation}
with  the Dirichlet boundary condition
\begin{equation}\label{bihar1dbd}
v(0)=v'(0)=v(1)=v'(1)=0
\end{equation}
We discretize the equation on the uniform mesh $0=x_0 \le x_1 \le \cdots \le x_N \le x_{N+1}=1$ with $x_i = ih$ and $h=1/(N+1)$. Let $\bv :=[ v(x_1),    v(x_2), \cdots,  v(x_N)]^T$.
A standard discretization  corresponding to the 13-point scheme  \cite{ehrl71} uses the second order center difference  for $ {d^2  \over dx^2}$ and  $v(-h) = v(h)+{\cal O} (h^3)$ for  the boundary condition $v'(0)=0$ with a similar one for $v'(1)=0$. This results in $A_h \bv = \lambda \bv + e$, where $\lambda$ is the eigenvalue of the operator (\ref{bihar1d}) and
\begin{equation}\label{bihar1dstandard}
A_h =\frac{1}{h^4}
   \left( \begin{array}{rrrrrrrr}
        7 & -4 & 1 &   &   &     &   & \\
        -4 & 6 & -4& 1 &   &     &   & \\
        1 & -4 & 6 & -4 & 1 &     &   & \\
          & 1 & -4 & 6 & -4 & 1 &     & \\
          &   & \ddots  & \ddots  & \ddots  & \ddots  &  \ddots &     \\
          &   &   &  1 &  -4 & 6  &  -4 & 1    \\
          &   &   &   &  1 &  -4 & 6  &  -4   \\
          &   &   &   &   &  1 &  -4 &  7
      \end{array} \right),\;\;
\be =
\left(
  \begin{array}{c}
    {\cal O} (h^{-1})\\ {\cal O} (h^2)\\ \vdots\\ {\cal O} (h^2)\\ {\cal O} (h^{-1})
  \end{array}
\right).
\end{equation}
We note that the local truncation error $e$ is of order ${\cal O} (h^{-1})$ only at the two boundary points, but it is shown in \cite{bram71,kurk} that this is sufficient to get a second order convergence. Namely, there is some eigenvalue of $A_h$ such that $|\lambda_h - \lambda| \le C h^2$ where $C$ is a constant dependent on the eigenvector $v$.

The standard discretization $A_h$ has a condition number of order $h^{-4}$. It is not diagonally dominant and it does not appear to have a factorization as a product of diagonally dominant matrices. Thus, direct computations of a few smallest eigenvalue of $A_h$ will have low accuracy when $h$ becomes small; see Example 4.

To be able to accurately compute a few smallest eigenvalues, we now derive a discretization that is a product of diagonally dominant matrices. This is done by exploring the product form $ {d^2  \over dx^2} {d^2  \over dx^2}$ of the operator.
Let $w= - {d^2 v \over dx^2}$ and  $u= - {d^2 w \over dx^2} ={d^4 v \over dx^4} $. Let
\[
\bv =
\left(
  \begin{array}{c}
    v(x_1) \\
    v(x_2) \\
    \vdots \\
    v(x_N) \\
  \end{array}
\right),
\;\;
\bw =
\left(
  \begin{array}{c}
    w(x_0) \\
    w(x_1) \\
    \vdots \\
    w(x_N) \\
    w(x_{N+1})
  \end{array}
\right),\;
\bu =
\left(
  \begin{array}{c}
    u(x_1) \\
    u(x_2) \\
    \vdots \\
    u(x_N) \\
  \end{array}
\right)
\mbox{ and }
\hat \bu =
\left(
  \begin{array}{c}
      u(x_0) \\
    u(x_1) \\
    \vdots \\
    u(x_N) \\
    u(x_{N+1})
  \end{array}
\right).
\]
Then,  we have, for $1\le i \le N$,
\begin{eqnarray*}
w(x_i)  &=& \frac{-v(x_{i-1}) +2 v(x_i) - v(x_{i+1})}{h^2}  +\frac{1}{12} v^{(4)}(x_i) h^2+{\cal O} (h^4) \\
    &=& \frac{-v(x_{i-1}) +2 v(x_i) - v(x_{i+1})}{h^2}  +\frac{1}{12} u(x_i) h^2+{\cal O} (h^4)
\end{eqnarray*}
and
\begin{eqnarray*}
u(x_i)  &=& \frac{-w(x_{i-1}) +2 w(x_i) - w(x_{i+1})}{h^2}  +\frac{1}{12} w^{(4)}(x_i) h^2+{\cal O} (h^4) \\
    &=& \frac{-w(x_{i-1}) +2 w(x_i) - w(x_{i+1})}{h^2}  +\frac{1}{12} u''(x_i) h^2+{\cal O} (h^4).
\end{eqnarray*}
Among many possible discretizations of the boundary condition $v'(x_0)=v'(x_{N+1})=0$, we use the following scheme to maintain the product form of the operator:
\[
w(x_0) = \frac{2 v(x_{1}) -  v(x_2) }{h^2} +{\cal O} (h)
\mbox{ and }
w(x_{N+1}) = \frac{2 v(x_{N}) -  v(x_{N-1})}{h^2} +{\cal O} (h),
\]
which is derived, in the case of the left end, by expanding $v(x_{1})$, $v(x_2)$ about $x_0$ using  $v(x_0) =v'(x_0)=0$.
Then,
\[
\bw = \frac{1}{h^2} \widehat S_N \bv  +\frac{1}{12} \hat \bu h^2 +f
\]
with $f=[{\cal O} (h), {\cal O} (h^4), \cdots, {\cal O} (h^4), {\cal O} (h)]^T$,  and
\begin{equation}\label{uTn}
\bu = \frac{1}{h^2} \widehat T_N \bw  +{\cal O} (h^2)
\end{equation}
where
\[
\widehat S_N =
\left(
  \begin{array}{ccccc}
    2 & -1 &    &  &        \\
    \hline \\
      &   &   T_N &    & \\
      \\
 \hline
      &   &     & -1 & 2 \\
  \end{array}
\right), \;\;
T_N =\left( \begin{array}{rrrrr}
     2 & -1 &  { }  &{ }    & { }            \\
     -1 & 2 & -1 & { }   & { }           \\
    { }    & \ddots & \ddots & \ddots        & { }      \\
         & { }      & -1   & 2 & -1  \\
      & & { }      & -1   & 2
     \end{array} \right)
\]
are respectively  $(N+2)\times N$ and $N\times N$ matrices and
\[
\widehat T_N =\left( \begin{array}{cccccc}
   -1 &  2 & -1 &  { }  &{ }    & { }            \\
     & -1 & 2 & -1 & { }   & { }           \\
     & { }    & \ddots & \ddots & \ddots        & { }      \\
     & { }     & { }      & -1   & 2 & -1
     \end{array} \right) \in \R^{N\times (N+2)}
\]
Note that by applying (\ref{uTn}) to $\hat \bu$, we have $\frac{1}{  h^2} \widehat T_N \hat \bu = \bv^{(6)} +{\cal O} (h^2)$ where $\bv^{(6)} = [v^{(6)} (x_i)]_{i=1}^N$.  Hence,
\begin{eqnarray*}
\bu  
  & = & \frac{1}{h^4} \widehat T_N \widehat S_N \bv +\frac{1}{12 h^2}  \widehat T_N \hat \bu h^2  +\frac{1}{12 h^2}  \widehat T_N f +{\cal O} (h^2) \\
& = & \frac{1}{h^4} \widehat T_N \widehat S_N \bv   +f_1
\end{eqnarray*}
with
\begin{equation}\label{trerror}
f_1:=\frac{1}{12}  \widehat T_N \hat \bu +\frac{1}{12 h^2}  \widehat T_N f +{\cal O} (h^2)=[{\cal O} (h^{-1}), {\cal O} (h^2), \cdots, {\cal O} (h^2), {\cal O} (h^{-1})]^T,
\end{equation}
where we note that $ \widehat T_N \hat \bu = \bv^{(6)}h^2 +{\cal O} (h^4)$ and $\widehat T_N f =[{\cal O} (h), {\cal O} (h^4), \cdots, {\cal O} (h^4), {\cal O} (h)]^T$.
Thus,
\begin{equation}\label{eq:fdeig}
\frac{1}{h^4} \widehat T_N \widehat S_N \bv   +f_1 = \lambda \bv
\end{equation}
Omitting the $f_1 $ term, we obtain the following discretization
\begin{equation}\label{eq:diseig}
\widehat T_N \widehat S_N \bv_h = \lambda_h h^4 \bv_h.
\end{equation}
The so derived discretization is in a factorized form, but $\widehat T_N$ and $\widehat S_N$ are not square matrices. However, they can be reduced to a product of square matrices as follows.
\begin{eqnarray}
B_h := \widehat T_N \widehat S_N
& = & T_N^2 -
\left(
  \begin{array}{ccccc}
    2 & -1 &   &   &     \\
    \hline \\
      &   &    0 &    & \\
      \\
 \hline
      &   &   & -1 & 2 \\
  \end{array}
\right) \nonumber \\
  & = & T_N^2 - E_N T_N = S_N  T_N  \label{eq:bh}
\end{eqnarray}
where
\[
E_N =  \left(
  \begin{array}{ccccc}
    1 & 0 &    &      & \\
    \hline \\
      &   &    0 &   &  \\
      \\
 \hline
      &    &   & 0 & 1 \\
  \end{array}
\right)
\mbox{ and }
S_N =T_N -E_N =\left( \begin{array}{rrrrr}
     1 & -1 &     &{ }    & { }            \\
     -1 & 2 & -1 & { }   & { }           \\
      { }    & \ddots & \ddots & \ddots        & { }      \\
      { }     & { }      & -1   & 2 & -1 \\
      { }     & { }   &    & -1   & 1
     \end{array} \right)
\]
are both  $N\times N$ matrices. Now, both  $S_N$ and $T_N$ are diagonally dominant and can be factorized accurately using Algorithm \ref{alg:lu}.

Note that the local truncation error $f_1$ of (\ref{trerror}) for this discretization is of order $h^2$ except at the boundary points where it is of order $h^{-1}$. By \cite{bram71,kurk}, this is sufficient to imply a second order approximation of eigenvalues, namely, for any eigenvalue of the differential operator\footnote{This may appears impossible with only a finite number of eigenvalues of a discretization matrix but as explained by Keller \cite{kell}, the constant $C$ in the bound depends on the smoothness of $v$ and for a higher eigenvalue, their corresponding eigenvector is highly oscillatory and thus $C$ is large enough that the bound $Ch^2$ is only meaningful for correspondingly small $h$.}, there is some eigenvalue of $B_h$ such that $|\lambda_h - \lambda| \le C h^2$ where $C$ is a constant dependent on the eigenvector $v$ (specifically, $v^{(6)}$). Namely, the finite difference method considers any operator eigenvalue $\lambda$ and approximates its associated equation (\ref{eq:fdeig}) by (\ref{eq:diseig}), from which {\em some} eigenvalue $\lambda_h$ of $B_h$ is found to approximate $\lambda$. It is not given that the pairing of $\lambda$ and $\lambda_h$ follows any  ordering of the eigenvalue sequences. (In contrast, projection  methods such as finite elements have each eigenvalue of the discretization, say the $i$-th smallest, approximating the $i$-th smallest  eigenvalue of the operator as $h\rightarrow 0$ as implied from the minimax theorem.)
Then, $B_h$ may have eigenvalues that do not approximate any eigenvalue  of the differential operator as $h\rightarrow 0$. Such eigenvalues are called spurious eigenvalues. In our particular discretization, it is easy to show that $0$ is always an eigenvalue of $B_h$ of multiplicity 1 (the next theorem) and is therefore a spurious eigenvalue as the differential operator does not posses any zero eigenvalue. Thus, in actual computations, we need to deflate the zero eigenvalue and  compute the smallest nonzero eigenvalues of $B_h$.

\begin{theorem}\label{thm:bh}
$B_h =  S_N  T_N $ is diagonalizable with real eigenvalues and $0$ is a simple eigenvalue with $e=[1, 1, \cdots, 1]^T$ as a left eigenvector and $v_0 = T_N^{-1}e$ as a right eigenvector. Furthermore, letting $P_0 =I - v_0 e^T / e^T v_0$ be the spectral projection associated with the nonzero eigenvalue, then $B_h\left|_{{\cal R} (P_0)}\right.$, the restriction of $B_h$ on ${\cal R} (P_0)$  (the spectral subspace complementary to $\tspan \{v_0\}$), is invertible with its eigenvalues being the nonzero eigenvalues of $B_h$.
\end{theorem}
\begin{proof}
By writing $B_h =  S_N  T_N = T_N^{-1/2} (T_N^{1/2} S_N T_N^{1/2}) T_N^{1/2}$, $B_h$ is diagonalizable with real eigenvalues. It is straightforward to check that $B_h v_0 = 0$ and $e^T B_h=0$. It follows from ${\rm rank} (S_N)=N-1$ that $0$ is a simple eigenvalue of $B_h$. The invertibility of $B_h\left|_{{\cal R} (P_0)}\right. $ and its spectral decomposition follows from the eigenvalue decomposition of $B_h$.
\end{proof}

{\sc Remark:} The invertibility of $B_h\left|_{{\cal R} (P_0)}\right.$ implies that $B_h x =y$ has a unique solution $x\in {\cal R} (P_0) $ for any $y \in {\cal R} (P_0)$. In particular, the eigenvalues of the inverse are the reciprocals of the non-zero eigenvalues of $B_h$.  Also, the unique solution to $B_h x =y$ can also be expressed as $x=B_h^D y\in {\cal R} (P_0) $ using the Drazin inverse $B_h^D$.

To compute a few smallest nonzero eigenvalue of $B_h$, we can use the deflation by restriction  to ${\cal R} (P_0)$. Namely, we compute the eigenvalues of  $B_h\left|_{{\cal R} (P_0)}\right. $. Indeed, to compute them accurately, we compute a few largest eigenvalues of the inverse of  $B_h\left|_{{\cal R} (P_0)}\right. $, which can be computed accurately  as follows.

First note that using the accurate $LDL^T$ factorization of $T_N$,  $v_0= T_N^{-1}e$ can be computed accurately. To apply  the inverse of $B_h\left|_{{\cal R} (P_0)}\right. $ on a vector $ y \in {\cal R} (P_0)$, we need to solve
\[
B_h x =y\;\;\mbox{ for }\;\; x, y \in {\cal R} (P_0).
\]
Since $S_N$ is also diagonally dominant, we first compute an accurate $LDL^T$ factorization $S_N = L_s D_s L_s^T$. Indeed, for this matrix, we have the factorization exactly with
\begin{equation}\label{eq:ls}
L_s =\left( \begin{array}{rrrr}
     1 &     &     &{ }         \\
     -1 & 1 &  & { }          \\
      { }    & \ddots & \ddots &    \\
      { }     & { }      & -1   & 1
     \end{array} \right)
\;\;\mbox{ and }\;\;
D_s = {\rm diag}\{1, \cdots, 1, 0\}.
\end{equation}
Then  $L_s^T e = e_N$, which  also follows from $S_N e = 0$, where $e_N=[0, \cdots, 0, 1]^T$. Thus, for $y \in {\cal R} (P_0)$, we have $e_N^T (L_s^{-1} y) = e^T y = 0$; namely the last entry of $L_s^{-1} y$ is zero. Thus a particular solution to $B_h x =y$ that is not necessarily in ${\cal R} (P_0)$ can be obtained by solving $L_s x_1 =  y$, $D_s x_2 = x_1$ (by setting $x_2 (i) = x_1 (i)$ for $1\le i \le N-1$, and $x_2 (N)=0$), $L_s^T x_3 = x_2$ and  $T_N x = x_3$ (by using the accurate $LDL^T$ factorization of $T_N$).
Since ${\rm Ker}(B_h) = \tspan \{v_0\}$, the general solution to $B_h x =y$ is $x(t) = x+t v_0$.
Now,  $x(t)$  is in ${\cal R} (P_0)$  if and only if $t=-e^T x / e^T v_0$.
Thus, $x(t) = x- v_0 \frac{e^T x}{e^T v_0}=P_0 x $ is the unique  solution to $B_h x =y$ that is in ${\cal R} (P_0)$.
We summarize this  as the following algorithm.

\begin{algorithm}\label{alg:bhinv}
{\sc Compute $x =\left( B_h|_{{\cal R} (P_0)}\right)^{-1} y$} (solve $B_h x =y$ for $x, y \in {\cal R} (P_0)$)
\begin{tabbing}
123\=123\= 1236\= 12325\= 3333\=3333\= \kill
\>1 \> Input: $y \in {\cal R} (P_0)$; \\
\>2 \> Compute an accurate $LDL^T$ factorization $T_N = L_t D_t L_t^T$ by Algorithm \ref{alg:lu}; \\
\>3 \> Solve $T_N v_0 = e$;\\
\>4 \>Solve $L_s x_1 =  y$; $x_2=x_1$ and explicitly set $x_2 (N)=0$; ($L_s$ defined by (\ref{eq:ls}))\\
\>5 \>Solve $L_s^T x_3 = x_2$;\\
\>6 \>Solve $T_N x = x_3$ using $T_N = L_t D_t L_t^T$; \\
\>7 \> $x = x - v_0 e^T x / e^T v_0$.
\end{tabbing}
\end{algorithm}

{\sc Remarks:} Lines 2 and 3 only need to be computed once if we run the above algorithm multiple times for different $y$. In Line 4, although $x_1 (N)=0$ in theory, it may not be the case due to roundoffs. We therefore explicitly set $x_2 (N)=0$.

We now present some  numerical results.

\noindent {\bf Example 4: } \label{e3}
Consider the biharmonic eigenvalue problem (\ref{bihar1d}) with the Dirichlet boundary condition
(\ref{bihar1dbd}). The eigenvalues are not known exactly for this problem, but solving it as an initial value problem, we can reduce it to an algebraic equation transcendental in the eigenvalue parameter \cite{embry}. The root of the transcendental can then be computed using Mathematica's
build-in root-finding routine in high precision. The following is the computed result using  50 digits as obtained by M. Embree  \cite{embry}:
\[
\lambda_1 \approx 500.56390174043259597023906145469523385520808092739.
\]
We now compute $\lambda_1$ using the  difference schemes $A_h$ of (\ref{bihar1dstandard}) and $B_h$ of (\ref{eq:bh}). The eigenvalues of $A_h$ and $B_h$ are computed using  inverse iteration\footnote{
We have also carried out tests using the Lanczos algorithm with similar results.} with $A_h^{-1}$ computed using the Cholesky factorization 
and $\left( B_h|_{{\cal R} (P_0)}\right)^{-1}$ computed by Algorithm \ref{alg:bhinv}.  We list the computed results for mesh size $h=2^{-k}$ with $k=4,5, \cdots, 19$ in Table  \ref{tbl:ex3}. Also listed are the relative errors as computed using the $\lambda_1$ above.

\begin{table}[hbt] \small
\begin{center}
\caption{\small {Example 4}: approximations of
$\lambda_1$:
$\mu_1^{chol}$ - $A_h$ with Cholesky;
$\mu_1^{aldu}$ - $B_h$ with accurate $LDL^T$.
}
\begin{tabular}{c|c|c|c|c}
$h$ &
$\mu_1^{chol}$ of $A_h$ & $\frac{|\lambda_{1}-\mu_1^{chol}|}{\lambda_{1}}$
& $\mu_1^{aldu}$ of $B_h$ & $\frac{|\lambda_{1}-\mu_1^{aldu}|}{\lambda_{1}}$ \\
\hline
 6.3e-2       &    4.84875068679297440e2      &    3.1e-2      &   5.02539119245910290e2       &    3.9e-3 \\
\hline
 3.1e-2       &    4.96560468599189620e2      &    8.0e-3      &   5.01071514661422610e2       &    1.0e-3 \\
\hline
 1.6e-2       &    4.99557827787957420e2      &    2.0e-3      &   5.00691660365858750e2       &    2.6e-4 \\
\hline
 7.8e-3       &    5.00312055034394060e2      &    5.0e-4      &   5.00595894739436520e2       &    6.4e-5 \\
\hline
 3.9e-3       &    5.00500919775768520e2      &    1.3e-4      &   5.00571903322230300e2       &    1.6e-5 \\
\hline
 2.0e-3       &    5.00548153582656820e2      &    3.1e-5      &   5.00565902344106350e2       &    4.0e-6 \\
\hline
 9.8e-4       &    5.00559945766432860e2      &    7.9e-6      &   5.00564401904366210e2       &    1.0e-6 \\
\hline
 4.9e-4       &    5.00562428167766880e2      &    2.9e-6      &   5.00564026782234690e2       &    2.5e-7 \\
\hline
 2.4e-4       &    5.00570856305839580e2      &    1.4e-5      &   5.00563933000933900e2       &    6.2e-8 \\
\hline
 1.2e-4       &    5.00646588821035950e2      &    1.7e-4      &   5.00563909555575040e2       &    1.6e-8 \\
\hline
 6.1e-5       &    5.00733625655273670e2      &    3.4e-4      &   5.00563903694203870e2       &    3.9e-9 \\
\hline
 3.1e-5       &    5.48097497225735650e2      &    9.5e-2      &   5.00563902228892570e2       &    9.8e-10 \\
\hline
 1.5e-5       &    7.35072209632324980e2      &    4.7e-1      &   5.00563901862573060e2       &    2.4e-10 \\
\hline
 7.6e-6       &    1.98724756006599050e3      &    3.0e0      &   5.00563901770967450e2       &    6.1e-11 \\
\hline
 3.8e-6       &    2.91400428172778860e3      &    4.8e0      &   5.00563901748025440e2       &    1.5e-11 \\
\hline
 1.9e-6       &   -   &  -    &   5.00563901742273290e2       &    3.7e-12
\end{tabular}

\label{tbl:ex3}
\end{center}
\end{table}

We note that both discretizations converge in the order of $h^2$. For the standard discretization $A_h$, the eigenvalue errors decreases for $h$  up to $10^{-3}$ and starts to increase afterwards. When $h\approx 10^{-6}$, there is no accuracy left in the computed eigenvalue. For $h=1.9e-6$, the matrix was tested as being indefinite by the Cholesky factorization algorithm (marked by the ``-" entry in the table). For the new discretization $B_h$, the order $h^2$ convergence is maintained for $h$ up to when the eigenvalue reach full machine accuracy with the relative error in the order of $n \bu \approx 10^{-11}$.
It is also interesting to note that, even before the roundoff errors overtake the discretization errors in the standard scheme $A_h$,  the new discretization $B_h$  results in  an eigenvalue error one order of magnitude smaller.


Thus, we have accurately computed the smallest eigenvalues of a 1-dimensional biharmonic operator with the Dirichlet boundary condition by deriving a product form discretization. Our method can be applied to fourth order operator such as
\[
{\cal L} v := {d^2  \over dx^2} \left( p(x) {d^2 v \over dx^2}\right)
\;\;\mbox{ or }\;\;
{\cal L} v :={d^4 v \over dx^4} +p(x)  {d^2 v \over dx^2}.
\]
However, we have not been able to generalize it to problems of dimension higher than 1.
This appears to be a difficulty intrinsic to the biharmonic operator in 2-dimension or higher. It is known that  the   biharmonic operator with the Dirichlet boundary condition do not have a so-called {\em positivity preserving property} in almost all kinds of  domains in dimension 2 or higher; see \cite{codu50,gara51,grro10} for details. The only exception is when the domain is a ball in $\R^n$ \cite{bogg05}, which includes $[0, 1]$ of  $\R$.  Notice that the two factors of $B_h$ are diagonally dominant M-matrices, from which it follows that $B_h$  has the positivity preserving property. With the operator in 2 dimension not having the positivity preserving property, it appears difficult to derive a discretization of the form of $B_h$.


\section{Concluding Remarks}
We have presented a new method  to compute a few smallest eigenvalues of large and extremely ill-conditioned matrices that are diagonally dominant or are products of diagonally dominant matrices. This can be used to compute eigenvalues of finite difference discretization of certain differential operators. In particular, the eigenvalues of the 1-dimensional biharmonic operator is accurately computed by deriving a  new discretization that can be written as a product of  diagonally dominant matrices. Unfortunately, it appears difficult to apply our present techniques to the  2-dimensional problems.
For the future works, it will be interesting to further investigate if there is a suitable generalization  to the 2-dimensional biharmonic problems. It will also be interesting to study if the present techniques can be used for other discretization methods  as well as adaptive techniques \cite{daxu08,mehr11} for differential operators.

\bigskip

{\bf Acknowledgement}: I would like to thank Professors Russell Carden, Mark Embree, and Jinchao Xu  for their interests and  discussions. In particular, Mark Embree provided me with his computed eigenvalue in high precision  in Example 4 and Russell Carden pointed me to the reference \cite{grro10}. I would also like to thank two anonymous referees for their many constructive comments that have improved the presentation of the paper.

%

\end{document}